\documentclass[a4paper, 11pt]{article}
\usepackage{amsmath}
\usepackage{amssymb,esint}
\usepackage{amscd}
\usepackage{xspace}
\usepackage{fancyhdr}
\newtheorem{theorem}{Theorem}[section]

\newtheorem{lemma}[theorem]{Lemma}

\newtheorem{remark}[theorem]{Remark}

\newenvironment{proof}[1][Proof]{\textbf{#1.} }{\hfill\rule{0.5em}{0.5em}}
{\catcode`\@=11\global\let\AddToReset=\@addtoreset
\AddToReset{equation}{section}

\AddToReset{theorem}{section}

\def\bq{\begin{eqnarray}}
\def\eq{\end{eqnarray}}
\def\bqq{\begin{eqnarray*}}
\def\eqq{\end{eqnarray*}}

\def \supp {\rm supp}

\begin{document}
\title{ Brezis-Gallouet-Wainger type inequality with critical  fractional Sobolev space and BMO}
\author{
 {\bf  Nguyen-Anh Dao \thanks{ Applied Analysis Research Group, Faculty of Mathematics and Statistics,
  		Ton Duc Thang University, Ho Chi Minh City, Vietnam. Email address: daonguyenanh@tdt.edu.vn } , ~Quoc-Hung Nguyen\thanks{Scuola Normale Superiore, Centro Ennio de Giorgi, Piazza dei Cavalieri 3, I-56100 Pisa, Italy.  Email address: quoc-hung.nguyen@sns.it}}\\[0.5mm]
} 
\maketitle
          \textbf{Abstract.}
          In this paper, we prove    the Brezis-Gallouet-Wainger type inequality involving  the BMO norm, the fractional Sobolev norm, and the logarithmic norm of $\mathcal{\dot{C}}^\eta$, for $\eta\in(0,1)$.

 \section{Introduction and main results} 
\hspace*{0.2in}  
The main purpose of this paper is to  established    $L^\infty$-bound by means of the BMO norm, or  the critical fractional Sobolev  norm with the logarithm of   $\mathcal{\dot{C}}^{\eta}$ norm. Such a $L^\infty$-estimate of this type  is known as  the Brezis-Gallouet-Waigner (BGW) type inequality.  Let us remind that Brezis-Gallouet \cite{BreGa}, and Brezis-Wainger \cite{BreWa} considered the relation between $L^\infty$, $W^{k,r}$, and $W^{s,p}$, and proved that there holds
\begin{equation}\label{1}
\|f\|_{L^\infty}\leq C\left(   1+  \log^\frac{r-1}{r}  \left(1+\|f\|_{W^{s,p}}\right) \right), \quad sp>n
\end{equation}
provided $\|f\|_{W^{k,r}}\leq 1$,  for $kr=n$. Its application is to prove the existence of solutions of the nonlinear Schr\"odinger equations, see details in \cite{BreGa}. We also note that an alternative proof of \eqref{1} was given by H. Engler \cite{Engler} for any bounded set in $\mathbb{R}^n$ with the cone condition. Similar embedding for vector functions $u$ with 
$div u=0$ was investigated by Beale-Kato-Majda:
\begin{equation}\label{1d}
\|\nabla u\|_{L^\infty} \leq C\left(1+\|\text{rot} u\|_{L^\infty}  \left(1+
\log(1+\|u\|_{W^{s+1,p}}) \right)+ \|\text{rot} u\|_{L^2} \right),
\end{equation}
for $sp>n$, see \cite{BeKaMa}  (see also  \cite{OgaTan} for an improvement of \eqref{1d}  in a bounded domain). An application of \eqref{1d} is to prove the breakdown of smooth solutions for the $3$-D Euler equations.  After that, estimate \eqref{1d} was enhanced by Kozono and  Taniuchi  \cite{KoTa} in that  $\|\text{rot} u\|_{L^\infty}$ can be relaxed to  $\|\text{rot} u\|_{BMO}$:
\begin{equation}\label{1e}
\|\nabla u\|_{L^\infty} \leq C\left(1+\|\text{rot} u\|_{BMO}  \left(1+
\log(1+\|u\|_{W^{s+1,p}}) \right)\right).
\end{equation}
To obtain \eqref{1e}, Kozono-Taniuchi \cite{KoTa} proved a logarithmic Sobolev inequality in terms of BMO norm and Sobolev norm that for any $1<p<\infty$, and for  $s>n/p$, then there is a constant $C=C(n, p, s)$ such that the estimate
\begin{equation}\label{2}
\|f\|_{L^\infty}\leq C\left(   1+  \|f\|_{BMO}  \left(1+\log^+ \left(\|f\|_{W^{s,p}}\right)\right) \right)
\end{equation}
holds for all $f\in W^{s,p}$. 
 Obviously, \eqref{2} is a generalization of   \eqref{1}. 
 \\
 
 Besides, it is interesting to note that  Gagliardo-Nirenberg type inequality with critical Sobolev space directly yields  BGW type inequality. For example,  H. Kozono, and H. Wadade \cite{KoWa}  proved the Gagliardo-Nirenberg type inequalities for the critical case and the limiting case of Sobolev space  as follows:
 \begin{equation}\label{1a}
 \|u\|_{L^q} \leq  C_n  r' q^{\frac{1}{r'}} \|u\|^\frac{p}{q}_{L^p}\|(-\Delta)^\frac{n}{2r} u\|^{1-\frac{p}{q}}_{L^r},
 \end{equation}  
 holds for all $u\in L^p\cap  \dot{H}^{\frac{n}{r},r} $
 with  $1\leq p < \infty$, $1<r<\infty$, and for all $q$ with $p\leq q<\infty$ (see also Ozawa \cite{Ozawa1}). 
 \\
And
 \begin{equation}\label{1b}
  \|u\|_{L^q} \leq  C_n  q \|u\|^\frac{p}{q}_{L^p}\|u\|^{1-\frac{p}{q}}_{BMO},
  \end{equation}
  holds for all $u\in L^p\cap  BMO$
   with  $1\leq p < \infty$, and for all $q$ with $p\leq q<\infty$. 
   \\
 As a result, \eqref{1a} implies
  \begin{equation}\label{1c}
 \|u\|_{L^\infty}\leq  C\left( 1+
 (\|u\|_{L^p} + \|(-\Delta)^\frac{n}{2r} u\|_{L^r})\left(\log(1+ \|(-\Delta)^\frac{s}{2} u\|_{L^q})\right)^\frac{1}{r'}
 \right), 
  \end{equation}
 for every  $1\leq p<\infty$, $1<r <\infty$, $1<q<\infty$  and $n/q<s <\infty$.
 \\
  While \eqref{1b} yields
 \begin{equation}\label{1f}
 \|u\|_{L^\infty}\leq  C\left( 1+
  (\|u\|_{L^p} + \|u\|_{BMO}  )\log(1+ \|(-\Delta)^\frac{s}{2} u\|_{L^q})
  \right),
 \end{equation} 
    for every  $1\leq p<\infty$, $1<q <\infty$,   and $n/q<s <\infty$.
    \\
Thus,  \eqref{1c} and \eqref{1f} may be regarded  as a generalization of BGW inequality.   Note that  
   in \eqref{1c} and \eqref{1f}, the logarithm term only contains  the semi-norm $\|u\|_{\dot{W}^{s,p}}$. 
   \\
 Furthermore, Kozono, Ogawa, Taniuchi \cite{KogaTa}  proved the logarithmic Sobolev inequalities in Besov space, generalizing the BGW inequality and the Beale-Kato-Majda inequality. 
\\

Motivated by these above results, in this paper, we  study BGW type inequality  by means of the BMO norm,  the fractional Sobolev norm and    the $\mathcal{\dot{C}}^\eta$ norm,  for $\eta\in(0,1)$. 
    Then,
our first result is as follows:
\begin{theorem} \label{thm2} 
 Let $\eta\in(0,1)$, and $\alpha\in(0,n)$. Then, there exists a constant  $C=C(n,\eta)>0$ such that the estimate 
\begin{align}\label{es1}
\|f\|_{L^\infty}\leq C+C\|f\|_{BMO}\left(1+\log^+\left[\sup_{z\in \mathbb{R}^n}\int_{\mathbb{R}^n}\frac{|f(y)|}{(|z-y|+1)^\alpha}dy +\|f\|_{\mathcal{\dot{C}}^\eta}\right]\right)
\end{align}
 holds for all $f\in  \mathcal{\dot{C}}^{\eta}\cap BMO$. We accept the notation $\log^+ s=\log s$ if $s\geq1$, and  $\log^+ s=0$ if $s\in(0,1)$.

 \end{theorem}
 \begin{remark} It is clear that $\left(\displaystyle\sup_{z\in \mathbb{R}^n}\int_{\mathbb{R}^n}\frac{|f(y)|}{(|z-y|+1)^\alpha}dy\right)$ is finite if $f\in L^1$. On the other hand,    if $f\in L^r$, $r> 1$, then
 	for any $\alpha\in (\frac{n}{r},n)$, we have  
 	\[
 	\displaystyle\sup_{z\in \mathbb{R}^n}\int_{\mathbb{R}^n}\frac{|f(y)|}{(|z-y|+1)^\alpha}dy\leq C \|f\|_{L^r},\]
 	where the constant $C$ is independent of $f$. 
 \end{remark}
 \begin{remark} If $\supp ~f \subset B_R$, then \eqref{es1} implies
 	\begin{align}\label{es1'}
 	\|f\|_{L^\infty}\leq C+C\|f\|_{BMO}\left(1+\log^+\left[R^{n-\alpha+\eta}+\|f\|_{\mathcal{\dot{C}}^\eta}\right]\right).
 	\end{align}
 \end{remark}
 \begin{remark}\label{re4}
 	Note that if $f\in W^{s,p}$ with $sp>n$, then  \eqref{es1} is stronger than \eqref{2}  since   $W^{s,p}\subset\mathcal{C}^{0,\eta}\subset\mathcal{\dot{C}}^{\eta}$, with $\eta=\frac{sp-n}{p}$.
 \end{remark} 
 Concerning  the  BGW type inequality involving the fractional Sobolev space, we have the following result:
 \begin{theorem}
 \label{thm4} Let $s>0,p\geq 1$ be such that $sp=n$. Let $\alpha>0$, $\eta\in (0,1)$. Then, there exists a constant  $C=C(n,s,p,\eta, \alpha)>0$ such that the estimate
  	\begin{align}\label{31}
  \|f\|_{L^\infty}\leq C+C  \|f\|_{\dot W^{s,p}} \left(1+\left(\log^+\left(\sup_{z\in \mathbb{R}^n}\int_{\mathbb{R}^n}\frac{|f(y)|}{(|z-y|+1)^\alpha}dy +\|f\|_{\mathcal{\dot{C}}^{\eta}}\right)\right)^{\frac{p-1}{p}}\right)
  \end{align}
   holds for all $f\in  \mathcal{\dot{C}}^{\eta}\cap \dot W^{s,p}$, where  $\dot W^{s,p}$ is the homogeneous fractional Sobolev space, see its definition below.
    \begin{remark} As Remark \ref{re4}, we can see that \eqref{31} is stronger then  \eqref{1}. Furthermore, if $\supp ~f \subset B_R$, then \eqref{es1} implies
   	\begin{align}\label{es1''}
   	\|f\|_{L^\infty}\leq C+C\|f\|_{\dot W^{s,p}} \left(1+\left(\log^+\left[R^{n-\alpha+\eta}+\|f\|_{\mathcal{\dot{C}}^\eta}\right]\right)^{\frac{p-1}{p}}\right).
   	\end{align}
   \end{remark}
  \begin{remark} We consider $f_\delta(x)=-\log(|x|+\delta)\psi(|x|)$, where $\psi\in\mathcal{C}^1_c([0,\infty))$,  $0\leq \psi\leq 1$,  $\psi(|x|)=1$ if $|x|\leq \frac{1}{4}$, and  $\delta>0$ is small enough. It is not hard to see that for any $\delta>0$ small enough
  	\begin{align*}
  	||f_\delta||_{L^\infty(\mathbb{R}^d)}\sim|\log(\delta)|, ~~	||f_\delta||_{BMO(\mathbb{R}^d)}\sim 1,~~\|f_\delta\|_{\dot W^{\frac{n}{p},p}}\sim |\log(\delta)|^{\frac{1}{p}},
  	\end{align*}
  	and \begin{align*}
  	\sup_{z\in \mathbb{R}^n}\int_{\mathbb{R}^n}\frac{|f_\delta(y)|}{(|z-y|+1)^\alpha}dy \sim 1,~~\|f_\delta\|_{\mathcal{\dot{C}}^{\eta}(\mathbb{R}^n)}\lesssim \delta^{-1}.
  	\end{align*}
  	Therefore, the power $1$ and  $\frac{p-1}{p}$ of the term  $\displaystyle\log_2^+\left(\sup_{z\in \mathbb{R}^n}\int_{\mathbb{R}^n}\frac{|f(y)|}{(|z-y|+1)^\alpha}dy +\|f\|_{\mathcal{\dot{C}}^{\eta}}\right)$      in \eqref{es1} and \eqref{31} respectively are sharp that there are no such estimates of the form: 
  	\begin{align*}
  	\|f_1\|_\infty\leq C+C\|f_1\|_{BMO}\left(1+\left(\log_2^+\left(\sup_{z\in \mathbb{R}^n}\int_{\mathbb{R}^n}\frac{|f_1(y)|}{(|z-y|+1)^\alpha}dy +\|f_1\|_{\mathcal{\dot{C}}^\eta}\right)\right)^\gamma\right),
  	\end{align*}
  	and 
  	\begin{align*}
  	\|f_2\|_{L^\infty}\leq C+C  \|f_2\|_{\dot W^{\frac{n}{p},p}} \left(1+\left(\log^+\left(\sup_{z\in \mathbb{R}^n}\int_{\mathbb{R}^n}\frac{|f_2(y)|}{(|z-y|+1)^\alpha}dy +\|f_2\|_{\mathcal{\dot{C}}^{\eta}}\right)\right)^{\gamma\frac{p-1}{p}}\right),
  	\end{align*}
  	hold for all $f_1\in BMO\cap \mathcal{{\dot{C}}^\eta}$, $f_2\in  \mathcal{\dot{C}}^{\eta}\cap \dot W^{s,p}$, for some $\gamma\in(0,1)$.
  \end{remark}

 \end{theorem}

Before closing this section, let us introduce some functional spaces that we use through this paper.
First of all, we recall $\mathcal{\dot{C}}^{\eta}$, $\eta\in(0,1)$, as the homogeneous Holder continuous of order $\eta$, endowed with the semi-norm:
\[
\|f\|_{\dot{\mathcal{C}}^\eta} = \sup_{x\not= y}\frac{|f(x)-f(y)|}{|x-y|^\eta}.
\]
Next, if $s\in(0,1)$, then  we recall $\dot{W}^{s,p}$  the homogeneous fractional Sobolev  space, endowed with the semi-norm:
\[
\|f\|_{\dot{W}^{s,p}} = \left( \int_{\mathbb{R}^n}\int_{\mathbb{R}^n}\frac{|f(x)-f(y)|^p}{|x-y|^{n+s p}} dxdy\right)^\frac{1}{p}.
\] 
When $s>1$, and $s$ is not an integer, we denote   $\dot{W}^{s,p}$ as the homogeneous fractional Sobolev space   endowed with the semi-norm:
\[
\|f\|_{\dot{W}^{s,p}}= 
\sum_{ |\sigma|=[s] }   \|D^\sigma f  \|_{\dot{W}^{s-[s],p}} .
\]
If  $s$ is an integer, then  \[
\|f\|_{\dot{W}^{s,p}}= 
\sum_{ |\sigma|=[s] }   \|D^\sigma f  \|_{L^p}  .
\]
We refer to \cite{ElePaVal} for details on the fractional Sobolev space.
\\
After that, we accept the notation  $(f)_\Omega:=\displaystyle\fint_\Omega f=\frac{1}{|\Omega|} \int_\Omega f(x)dx$ for any Borel set $\Omega$. Finally,  $C$ is always denoted as a constant which can change from line to line. And $C(k,n,l)$ means that this constant merely depends on $k,n,l$.   

\section{Proof of the Theorems} 
 We first prove Theorem \ref{thm2}.
 \vspace*{0.1in}
 \\
\begin{proof}[Proof of Theorem \ref{thm2} ] It is enough to prove that 
		\begin{align}\label{es1re}
	|f(0)|\leq C+C\|f\|_{BMO}\left(1+\log_2^+\left(\int_{\mathbb{R}^n}\frac{|f(y)|}{(|y|+1)^\alpha}dy +\|f\|_{\mathcal{\dot{C}}^\eta}\right)\right),
	\end{align}
	 Let $m_0\in \mathbb{N}$, set $B_\rho:=B_\rho(0)$,  we have, 
	 \begin{align*}
	 |f(0)|&=\left|f(0)-\fint_{B_{2^{-m_0}}}f+\sum_{j=-m_0}^{m_0-1}\left(\fint_{B_{2^{j}}}f-\fint_{B_{2^{j+1}}}f\right)+
	 \fint_{B_{2^{m_0}}}f\right|\\& \leq \fint_{B_{2^{-m_0}}}|f-f(0)|+\sum_{j=-m_0}^{m_0-1}\fint_{B_{2^{j}}}|f-(f)_{B_{2^{j+1}}}|+
	 C2^{-m_0(n-\alpha)} \int_{B_{2^{m_0}}}\frac{|f(y)|}{(|y|+1)^\alpha}dy
	 \\& \leq \fint_{B_{2^{-m_0}}}|y|^\eta \|f\|_{\mathcal{\dot{C}}^\eta} dy+2m_0 
	 \|f\|_{BMO}+
	 C2^{-m_0(n-\alpha)} \int_{B_{2^{m_0}}}\frac{|f(y)|}{(|y|+1)^\alpha}dy
	 	\\ & \leq C2^{-m_0\min\{n-\alpha,\eta\}} \left(\int_{\mathbb{R}^n}\frac{|f(y)|}{(|y|+1)^\alpha}dy +\|f\|_{\mathcal{\dot{C}}^\eta}\right)+Cm_0 
	 \|f\|_{BMO}.
	 \end{align*}
	Choosing \[m_0=\left[\frac{\log_2^+\left(\displaystyle \int_{\mathbb{R}^n}\frac{|f(y)|}{(|y|+1)^\alpha}dy +\|f\|_{\mathcal{\dot{C}}^{\eta}}\right)}{\min\{n-\alpha,\eta\}}\right]+1,\] we get \eqref{es1re}. The proof is complete.
\end{proof}\\ 
 
Next, we prove Theorem \ref{thm4}.
\vspace*{0.1in}
\\
\begin{proof}[Proof  of Theorem \ref{thm4}] To prove it, we need the following lemmas:
\begin{lemma}\label{lem2}
  Let $a_0=1$, and  let  $(a_0, a_1,...,a_{k+1})\in\mathbb{R}^{k+2}$, for any $k\geq 1$, be  a unique solution of the following system:   \begin{equation}\label{eq3}
  \displaystyle\sum_{j =0}^{k+1} a_j 2^{jl}=0, \quad \forall l=0,...,k.
  \end{equation}
Then, we have
		\begin{align}\label{eq2}
		a:=\sum_{j=0}^{k}  (k-j+1) a_j \not =0.
		\end{align}
		Moreover,  for any $m\geq 1$, and for 
		$b, b_l \in\mathbb{R}$, $l=-m,...,m$, we have 
		\begin{align}\label{eq1}
		\sum_{l=-m}^{m-1}\left[\sum_{j=0}^{k+1} a_j b_{j+l}\right]= \sum_{l=m}^{k+m} \left[ \sum_{j=l-m+1}^{k+1} a_j \right] b_{l}  + \sum_{l=-m}^{k-m} \left[ \sum_{j=0}^{l+m} a_j \right]( b_{l}-b)+ ab.
		\end{align}
		As a result, we obtain
		\begin{equation}\label{eq5}
		|b|\leq \frac{1}{|a|} \left[ \sum_{j=0}^{k+1} |a_j| \right] \sum_{l=-m}^{k-m} |b_{l}-b|+\frac{1}{|a|} \sum_{l=-m}^{m-1}\left|\sum_{j=0}^{k+1} a_j b_{j+l}\right|+\frac{1}{|a|} \left[ \sum_{j=0}^{k+1} |a_j| \right]\sum_{l=m}^{k+m} | b_{l}|.
		\end{equation}
\end{lemma}
\begin{proof}
First of all, we note that  $a_{j}\not=0$, for $j=0,...,k+1$. 
Set
\[Q(x)= \sum_{j=0}^{k+1} a_{j}x^{j}.\] Then, 
\[
Q^\prime(1)= \sum_{j=1}^{k+1} ja_j.
\]
On the other hand, by \eqref{eq3}, we have $Q(2^l)=0$, for $l=0,...,k$. Thus, 
\[Q(x)=a_{k+1}\prod_{l=0}^{k}  (x-2^l),~ \text{ and  } Q^\prime(1)=\prod_{l=1}^{k}  (1-2^l) .\]
This implies 
\begin{equation}\label{eq4}
\sum_{j=1}^{k+1} ja_j=\prod_{j=1}^{k}  (1-2^l)\not=0.
\end{equation}
Next, we observe that
\[0=(k+1)\sum_{j=0}^{k+1} a_j =  a+  \sum_{j=1}^{k+1} j a_j=0.\]
The last equation and $\eqref{eq4}$ yield  $a=-\displaystyle\prod_{j=1}^{k}  (1-2^l)\not=0$. 
\\

Now, we prove $\eqref{eq1}$. We denote  $LHS$ (resp. $RHS$) is the left hand side (resp. the right hand side) side  of \eqref{eq1}.
It is not difficult to verify that
\[
\sum_{l=-m}^{k-m} \left[ \sum_{j=0}^{l+m} a_j \right]b= ab.
\] 
Then, a direct computation shows
\begin{align*}
& RHS= a_0 b_{-m}+ (a_0+a_1)b_{1-m} +...+(a_0+...+a_k)b_{k-m} 
\\
& ~~~~~~~~+(a_1+...+a_{k+1})b_m + (a_2+...+a_{k+1})b_{m+1}+...+ a_{k+1}b_{k+m}= a_0 \sum_{l=-m}^{k-m}b_l 
\\
& ~~~~~~~~+a_1 \left( \sum_{l=1-m}^{k-m}b_l +\sum_{l=m}^{m}b_l\right) +...+a_{k}  \left( \sum_{l=k-m}^{k-m} b_{k-m}+\sum_{l=m}^{m+k-1}b_l\right) + a_{k+1}  \left(\sum_{l=m}^{m+k}b_l\right) .
\end{align*}
Note that  $\displaystyle\left(\sum_{j=0}^{k+1} a_j\right) \sum_{l=k+1-m}^{m-1} b_l=0$.  Thus,
\begin{align*}
& RHS= RHS+ \left(\sum_{j=0}^{k+1} a_j\right) \sum_{l=k+1-m}^{m-1} b_l 
\\
&~~~~~~=\sum_{j=0}^{k+1} a_j\left(\sum_{l=j-m}^{j+m-1}b_l \right)
\\
&~~~~~~=\sum_{l=m}^{k+m} \left( \sum_{j=l-m+1}^{k+1} a_j \right) b_{l}  +\sum_{l=k+1-m}^{m-1} \left( \sum_{j=0}^{k+1} a_j \right) b_{l}  + \sum_{l=-m}^{k-m} \left( \sum_{j=0}^{l+m} a_j \right)b_{l}
\\
&~~~~~~=\sum_{l=m}^{k+m} \left( \sum_{j=l-m+1}^{k+1} a_j \right) b_{l}  + \sum_{l=-m}^{k-m} \left( \sum_{j=0}^{l+m} a_j \right)b_{l}
\\
&~~~~~~ =LHS.
\end{align*}
Or, we get $\eqref{eq1}$.
\\
Finally, \eqref{eq5} follows from \eqref{eq1} by using the triangle inequality.
In other words, we get Lemma \ref{lem2}.
\end{proof}
\\

Next, we have
\begin{lemma}\label{lem3} Assume  $a_0, a_1,...,a_{k+1}$ as in Lemma \ref{lem2}. Let $\Omega_j=B_{2^{j+1}} \backslash B_{2^j}$, where $B_\rho:=B_\rho(0)$ for any $\rho>0$.
Then, there holds
\begin{align}\label{es}
\left|\sum_{j=0}^{k+1} a_j\fint_{\Omega_j} f\right|\leq C\fint_{B_{2^{k+3}}\backslash B_{2^{-1}}  }  \left| D^{k} f(y)-\left(D^{k} f\right)_{ B_{2^{k+3}}\backslash B_{2^{-1}} } \right|dy.
\end{align}
For any $l\in\mathbb{R}$, we set $E_l=B_{2^{k+l+3}}\backslash B_{2^{l-1}}$.  As a consequence of \eqref{es}, we obtain 
\begin{align}\label{es"}
\left|\sum_{j=0}^{k+1} a_j\fint_{\Omega_{j+l}} f \right|\leq C  2^{kl}\fint_{E_l} \fint_{E_l}  
\left|D^{k} f(y)-D^{k} f(y')\right|dy dy'.
\end{align}
Moreover,  by the triangle inequality we get from \eqref{es"}
\begin{align}\label{es'}
\left|\sum_{j=0}^{k+1} a_j\fint_{\Omega_{j+l}} f\right|\leq C ~ 2^{kl} \fint_{E_l} \left|D^{k} f(y)\right|dy.
\end{align}
\end{lemma}
\begin{proof}
 Assume a contradiction  that \eqref{es} is not true.  There exists then a sequence $(f_m)_{m\geq 1}\subset W^{k,1}(B_{2^{k+3}}\backslash B_{2^{-1}} ) $ such that 
\begin{align}\label{20}
\int_{B_{2^{k+3}}\backslash  B_{2^{-1}}}|D^{k} f_m(y)-\left(D^{k} f_m\right)_{B_{2^{k+3}}\backslash  B_{2^{-1}}}|dy\leq \frac{1}{m},
\end{align}
and 
\begin{align*}
|\sum_{j=0}^{k+1} a_j\fint_{\Omega_j} f_m|=1, ~~\forall m\geq 1.
\end{align*}
Let us put 
\begin{align*}
\tilde{f}_m(x)=f_m(x)-P_{k,m}(x), ~~\text{with } P_{k,m}(x)=\sum_{l=0}^{k}\sum_{\alpha_1+...+\alpha_n=l}c_{l, k, m}(\alpha_1,...,\alpha_n)x_1^{\alpha_1}x_2^{\alpha_2}... x_n^{\alpha_n},
\end{align*}
and $c_{l, k, m}(\alpha_1,...,\alpha_n)$ is a constant such that 
\begin{equation}\label{21}
\left(D^l\tilde{f}_m\right)_{B_{2^{k+3}}\backslash  B_{2^{-1}}}=0, ~\forall l=0,...,k.
\end{equation}
By a sake of brief, we denote  $c_{l, m}=c_{l, k, m}(\alpha_1,...,\alpha_n)$.
Since $P_{k,m}$ is a polynomial of at most $k$-degree, then   $D^k P_{k,m}=const$. This fact, \eqref{20},  and $\eqref{21}$ imply 
 
\begin{align*}
\int_{B_{2^{k+3}}\backslash  B_{2^{-1}}}|D^{k} \tilde{f}_m (y) |dy= \int_{B_{2^{k+3}}\backslash  B_{2^{-1}}}|D^{k} f_m (y) -\left(D^{k} f_m\right)_{B_{2^{k+3}}\backslash  B_{2^{-1}}}|dy \leq \frac{1}{m}.
\end{align*}
It follows  form the compact embeddings  that there exists a subsequence of $(\tilde{f}_m)_{m\geq 1}$ (still denoted as $(\tilde{f}_m)_{m\geq 1}$) such that that $\tilde{f}_m\to \tilde{f}$ strongly in $L^1(B_{2^{k+3}}\backslash  B_{2^{-1}})$, and 
\begin{align*}
D^{k} \tilde{f}=0, ~~\text{ in}~~B_{2^{k+3}}\backslash  B_{2^{-1}}.
\end{align*}
This implies that $\tilde{f}$ is a polynomial of at most $(k-1)$-degree, i.e: 
\begin{align*}
\tilde{f}(x)=\sum_{l=0}^{k-1}\sum_{\alpha_1+...+\alpha_n=l}c_{l,k}'(\alpha_1,...,\alpha_n)x_1^{\alpha_1}x_2^{\alpha_2}... x_n^{\alpha_n},~~\forall x\in B_{2^{k+3}}\backslash  B_{2^{-1}}.
\end{align*}
On the other hand, we observe  that for any $l=0,...,k$
\begin{align*}
&\sum_{j=0}^{k+1} a_j \fint_{\Omega_j}  \sum_{\alpha_1+...+\alpha_n=l} c(\alpha_1,..., \alpha_n)  x_1^{\alpha_1}x_2^{\alpha_2}...x_n^{\alpha_n} dx_1 dx_2...dx_n\\&= \sum_{j=0}^{k+1} a_j \fint_{\Omega_1} \sum_{\alpha_1+...+\alpha_n=l} c(\alpha_1,..., \alpha_n)(2^jx_1)^{\alpha_1}(2^jx_2)^{\alpha_2}...(2^j x_n)^{\alpha_n} dx_1 dx_2... dx_n\\&= \fint_{\Omega_1} \sum_{\alpha_1+...+\alpha_n=l}c(\alpha_1,..., \alpha_n) \left(\sum_{j=0}^{k+1} a_j  2^{jl}\right)  x_1^{\alpha_1}x_2^{\alpha_2}... x_n^{\alpha_n} ~ dx_1dx_2...dx_n=0,
\end{align*}
  by $\eqref{eq3}$.
This implies
\begin{equation}\label{24}
\sum_{j=0}^{k+1} a_j\fint_{\Omega_j} \tilde{f} =0,
\end{equation}
and
\begin{align*}
\left|\sum_{j=0}^{k+1} a_j\fint_{\Omega_j} \tilde{f}_m\right|= \left|\sum_{j=0}^{k+1} a_j\fint_{\Omega_j} f_m \right|=1.
\end{align*}
Remind that $\tilde{f}_m \rightarrow \tilde{f}$ strongly in $L^1(B_{2^{k+3}}\backslash  B_{2^{-1}})$, then we have 
\begin{align*}
\left|\sum_{j=0}^{k+1} a_j\fint_{\Omega_j} \tilde{f}\right|=1.
\end{align*}
Or, we complete the proof of \eqref{es}.  
\\
The proof of \eqref{es"} (resp. \eqref{es'}) is trivial then we leave it to the reader.
 This puts an end to the proof of Lemma \ref{lem3}.
\end{proof}
\\

Now, we are ready to prove Theorem \ref{thm4}. \\
It is enough to show that 
\begin{align}\label{eq15b}
|f(0)|\leq C+C  \| f\|_{\dot W^{s,p}} \left(1+\log_2^+\left(\int_{\mathbb{R}^n}\frac{|f(y)|}{(|y|+1)^\alpha}dy +\|f\|_{\mathcal{\dot{C}}^{\eta}}\right)\right)^{\frac{p-1}{p}}.
\end{align}
\\
Set $s_1=s- k$,  $s_1\in [0,1)$. Then,  we divide our study into the two cases:
\\ 

{\bf i) Case: $s_1\in (0,1)$:}
\\
We apply Lemma \ref{lem2}  with  $b=f(0)$, $b_j=\displaystyle\fint_{\Omega_j}  f$. Then,  for any $m_0\geq 1$, there is a constant $C=C(k)>0$ such that
\begin{align}\label{eq7}
\begin{split}
|f(0)|\leq C \left(\sum_{l=-m_0}^{k-m_0}\left|\fint_{\Omega_l}f -f(0) \right|+ \sum_{l=-m_0}^{m_0-1}\left|\sum_{j=0}^{k+1} a_j\fint_{\Omega_{j+l}}   f\right|+  \sum_{l=m_0}^{k+m_0}  \left|\fint_{\Omega_l} f \right| \right).
\end{split}
\end{align}
Concerning the first term  on the right hand side of  \eqref{eq7}, we have
\begin{align*}
\begin{split}
\sum_{l=-m_0}^{k-m_0}\left|\fint_{\Omega_l}f -f(0) \right|\leq  \sum_{l=-m}^{k-m}  \fint_{\Omega_l} |f-f(0) | \leq   \sum_{l=-m}^{k-m}  \fint_{\Omega_l}  |x|^\eta \|f\|_{\mathcal{\dot{C}}^{\eta}}  dx.
\end{split}
\end{align*}
Thus,
\begin{equation}\label{eq7a}
\sum_{l=-m}^{k-m}\left|\fint_{\Omega_l}f -f(0) \right|\leq   \sum_{l=-m}^{k-m}  2^{(l+1)\eta}  \|f\|_{\mathcal{\dot{C}}^{\eta}} \leq C(\eta, k) 2^{-m\eta} \|f\|_{\mathcal{\dot{C}}^{\eta}}.
\end{equation}
Next, we use \eqref{es"} in Lemma \ref{lem3} to obtain
\begin{equation}\label{eq9}
\begin{split}
\sum_{l=-m}^{m-1}  \left|\sum_{j=0}^{k+1} a_j\fint_{\Omega_{j+l}}   f\right|   \leq C \sum_{l=-m}^{m-1}  2^{kl} \fint_{E_{l}} \fint_{E_{l}}   \left| D^{k} f(y)-D^k f(z)  \right| dy dz,
\end{split}
 \end{equation}
 where $E_l=B_{2^{k+l+3}}\backslash B_{2^{l-1}}$.
 	It follows from  H\"older's inequality
	\begin{align*}
&	\sum_{l=-m_0}^{m_0-1}2^{kl}\fint_{E_{l}}\fint_{E_{l}}|D^k f(y)- D^k f(z)|dydz \leq
	\\
 &\sum_{l=-m_0}^{m_0-1} 2^{kl}|E_l|^{-2} \left(\int_{E_{l}}\int_{E_{l}}\frac{|D^k f(y)- D^k f(z)|^p}{|y-z|^{n+s_1p}}dydz\right)^{\frac{1}{p}} 
 \left(\int_{E_{l}}\int_{E_{l}}|y-z|^{\frac{n+s_1 p}{p-1}}dydz\right)^{\frac{p-1}{p}}.
	\end{align*}
	Since $y, z\in E_l$, we have  $|y-z|\leq |y|+|z|\leq 2^{k+l+4}$. 
	Thus, the right hand side of the indicated inequality is less than
	\begin{align*}
		C(n,p,k)  2^{kl+\frac{l(n+s_1p)}{p}}
	|E_l|^{\frac{-2}{p}} \sum_{l=-m_0}^{m_0-1}  \left(\int_{E_{l}}\int_{E_{l}}\frac{|D^kf(y)-D^kf(z)|^p}{|y-z|^{n+s_1p}}dydz\right)^{\frac{1}{p}}.
	\end{align*}
	Note that $n=sp=(k+s_1)p$, and $|E_l|^{\frac{-2}{p}}\leq C(n,p,k) 2^{-2l\frac{n}{p}} $. 
	\\
	 Then,  there is a constant $C=C(k,s,n)>0$ such that  
	\begin{equation}\label{eq11}
	\sum_{l=-m_0}^{m_0-1}2^{kl}\fint_{E_{l}}\fint_{E_{l}}|D^k f(y)- D^k f(z)|dydz \leq C \sum_{l=-m_0}^{m_0-1}  \left(\int_{E_{l}}\int_{E_{l}}\frac{|D^kf(y)-D^kf(z)|^p}{|y-z|^{n+s_1p}}dydz\right)^{\frac{1}{p}}.
	\end{equation}
	 Thanks to the inequality 
	\begin{align}\label{es6}
	\displaystyle\sum_{j=-m_0}^{m_0-1}c_j^{\frac{1}{p}}\leq (2m_0)^{\frac{p-1}{p}} \left(\sum_{j=-m_0}^{m_0-1} c_j\right)^{\frac{1}{p}}, 
	\end{align}
we have
\begin{equation}\label{eq12}
\begin{split}
&\sum_{l=-m_0}^{m_0-1}  \left(\int_{E_{l}}\int_{E_{l}}\frac{|D^kf(y)-D^kf(z)|^p}{|y-z|^{n+s_1p}}dydz\right)^{\frac{1}{p}} 
\\ &~~~~~~~\leq (2m_0)^{\frac{p-1}{p}}\left(\sum_{l=-m_0}^{m_0-1}  \int_{E_{l}}\int_{E_{l}}\frac{|D^kf(y)-D^kf(z)|^p}{|y-z|^{n+s_1p}}dydz\right)^{\frac{1}{p}}.
\end{split}
\end{equation}
Moreover, we observe that  $\displaystyle\sum_{l=-\infty}^{+\infty}\mathbb{\chi}_{E_l\times E_l}(y_1,y_2)\leq k+4$, ~for all $(y_1,y_2)\in\mathbb{R}^n\times\mathbb{R}^n$. Thus,
\begin{equation}\label{eq13}
\sum_{l=-m_0}^{m_0-1}  \int_{E_{l}}\int_{E_{l}}\frac{|D^kf(y)-D^kf(z)|^p}{|y-z|^{n+s_1p}}dydz\leq (k+4) \int_{\mathbb{R}^n}\int_{\mathbb{R}^n}\frac{|D^kf(y)-D^kf(z)|^p}{|y-z|^{n+s_1p}}dydz. 
\end{equation}
Combining \eqref{eq11}, \eqref{eq12} and \eqref{eq13} yields
\begin{equation}\label{eq14}
\sum_{l=-m_0}^{m_0-1}2^{kl}\fint_{E_{l}}\fint_{E_{l}}|D^k f(y)- D^k f(z)|dydz\leq C(k,s,n) m_0^{\frac{p-1}{p}}\|f\|_{\dot{W}^{s,p}} .
\end{equation}
It remains to treat the last term. Then, it  is not difficult to see that  for any $\alpha>0$
\begin{align}\label{eq15}
\begin{split}
\sum_{l=m_0}^{k+m_0}  \left|\fint_{\Omega_l} f \right|&\leq C(k,n)2^{-m_0n}\int_{B_{2^{k+m_0}}}|f|
\\
&\leq C(k,n,\alpha) 2^{-m_0(n-\alpha)} \int_{B_{2^{k+m_0}}} \frac{|f(x)|dx}{(|x|+1)^\alpha} .
\end{split}
\end{align}
Inserting  \eqref{eq7a}, \eqref{eq14}, and \eqref{eq15} into \eqref{eq7} yields
\begin{align}\label{*}
|f(0)|&\leq C 2^{-m_0\min\{n-\alpha,\eta\}} \left(\int_{\mathbb{R}^n}\frac{|f(y)|}{(|y|+1)^\alpha}dy +\|f\|_{\mathcal{\dot{C}}^{\eta}}\right)+C m_0^{\frac{p-1}{p}} \|f\|_{\dot{W}^{s,p}}.
\end{align}
By choosing   
\[m_0=\left[\frac{\log_2^+\left(\displaystyle\int_{\mathbb{R}^n}\frac{|f(y)|}{(|y|+1)^\alpha}dy +\|f\|_{\mathcal{\dot{C}}^{\eta}}\right)}{\min\{n-\alpha,\eta\}}\right]+1
,\]
we obtain \eqref{eq15b}.\\
\medskip\\
%
%
	{\bf ii) Case: $s_1=0 ~(s=k)$:} 
	\\
	The proof is similar to the one of
	 the case $s_1\in(0,1)$. There  is just a  difference of estimating the second term on the right hand side of \eqref{eq7} as follows:
	\\
	
	Use \eqref{es'}, we get
\begin{equation}\label{eq16a}
 \sum_{l=-m_0}^{m_0-1}\left|\sum_{j=0}^{k+1} a_j\fint_{\Omega_{j+l}}   f\right|\leq 	C\sum_{l=-m_0}^{m_0-1} 2^{kl} \fint_{E_l} |D^{k} f|.\end{equation}
 Applying H\"older's inequality, we have
 \begin{align}\sum_{l=-m_0}^{m_0-1} 2^{kl} \fint_{E_l} |D^{k} f|&\leq  \nonumber \sum_{l=-m_0}^{m_0-1}  2^{kl} |E_l|^{-1/p} \left( \int_{E_l} |D^{k} f|^p\right)^{1/p}  
 \\ &\nonumber
 \leq C(n,k) \sum_{l=-m_0}^{m_0-1} \left( \int_{E_l} |D^{k} f|^p\right)^{1/p}
 \\ &
 \leq C m_0^{\frac{p-1}{p}}\left( \sum_{l=-m_0}^{m_0-1} \int_{E_l} |D^{k} f|^p\right)^{1/p}.\label{eq16}
 \end{align}
	We utilize the fact $\displaystyle\sum_{l=-\infty}^{\infty}\chi_{E_l}(y)\leq k+4,~\forall y\in \mathbb{R}^n$ again in order to get
	\begin{align}\label{eq17}
\left(\sum_{l=-m_0}^{m_0-1} \int_{E_l} |D^{k} f|^p\right)^{1/p}\leq (k+4)	\left( \int_{\mathbb{R}^n} |D^{k} f|^p\right)^{1/p}.
	\end{align}
From  \eqref{eq17}, \eqref{eq16}, and \eqref{eq16a}, we get
\begin{equation}\label{eq18}
\sum_{l=-m_0}^{m_0-1}\left|\sum_{j=0}^{k+1} a_j\fint_{\Omega_{j+l}}   f \right|\leq C(k,n) \| f\|_{\dot W^{s,p}}.
\end{equation}
	Thus, we obtain another version of \eqref{*} as follows: 
\begin{align}\label{**}
|f(0)|&\leq C 2^{-m_0\min\{n-\alpha,\eta\}} \left(\int_{\mathbb{R}^n}\frac{|f(y)|}{(|y|+1)^\alpha}dy +\|f\|_{\mathcal{\dot{C}}^{\eta}}\right)+C m_0^{\frac{p-1}{p}}  \| f\|_{\dot W^{s,p}}.
\end{align}
By the same argument as above (after \eqref{*}), we get the proof of the case $s_1=0$. This completes the proof of Theorem \ref{thm4}.
		\end{proof}

\end{document}